\documentclass[11pt]{amsart}
\usepackage{amsmath,amssymb,amsfonts}
\newtheorem{thm}{Theorem}[section]
\newtheorem{theorem}[thm]{Theorem}
\newtheorem {definition}[thm]{Definition}

\newtheorem{corollary}[thm]{Corollary}
\newtheorem{lemma}[thm]{Lemma}
\newtheorem{claim}{Claim}
\newtheorem{proposition}[thm]{Proposition}
\newtheorem{example}[thm]{Example}
\newtheorem{remark}[thm]{Remark}
\newtheorem* {problem}{Problem}

\numberwithin{equation}{section}
\begin{document}
\baselineskip=16pt
\title{Some remarks on the planar Kouchnirenko's Theorem}         
\author{Gert-Martin Greuel and Nguyen Hong Duc}
\address{Gert-Martin Greuel\newline \indent Universit\"{a}t Kaiserslautern, Fachbereich Mathematik, Erwin-Schr\"{o}dinger-Strasse, 
\newline \indent 67663 Kaiserslautern, Tel. +496312052850, Fax +496312054795}
\email{greuel@mathematik.uni-kl.de}
\address{Nguyen Hong Duc\footnote{Supported by DAAD (Germany) and NAFOSTED (Vietnam)}
\newline\indent Institute of Mathematics, 18 Hoang Quoc Viet Road, Cau Giay District \newline \indent  10307, Hanoi.} 
\email{nhduc@math.ac.vn}
\address{Universit\"{a}t Kaiserslautern, Fachbereich Mathematik, Erwin-Schr\"{o}dinger-Strasse,  
\newline \indent 67663 Kaiserslautern}
\email{dnguyen@mathematik.uni-kl.de}
\date{\today}            

\maketitle
\begin{abstract}
We consider different notions of non-degeneracy, as introduced by Kouchnirenko (NND), Wall (INND) and Beelen-Pellikaan (WNND) for plane curve singularities $\{f(x,y) = 0\}$ and introduce the new notion of weighted homogeneous Newton non-degeneracy (WHNND). It is known that the Milnor number $\mu$ resp. the delta-invariant $\delta$ can be computed by explicit formulas $\mu_N$ resp. $\delta_N$ from the Newton diagram of $f$ if $f$ is NND resp. WNND. It was however unknown whether the equalities $\mu=\mu_N$ resp. $\delta=\delta_N$ can be characterized by a certain non-degeneracy condition on $f$ and, if so, by which one. We show that $\mu=\mu_N$ resp. $\delta=\delta_N$ is equivalent to INND resp. WHNND and give some applications and interesting examples related to the existence of "wild vanishing cycles". Although the results are new in any characteristic, the main difficulties arise in positive characteristic. 
\end{abstract}

\section{Introduction}
Let $K$ be an algebraically closed field, $K[[x]]=K[[x_1,\ldots,x_n]]$ the formal power series ring and $\mathfrak m$ its maximal ideal. Let us recall the definition of the Newton diagram and Wall's notion of a $C$-polytope (see \cite{Wal99}). To each power series $f =\sum_{\alpha}c_\alpha x^\alpha\in K[[x]]$ we can associate its {\em Newton polyhedron} $\Gamma_+(f)$ as the convex hull of the set
$$\bigcup_{\alpha\in\text{supp}(f)}(\alpha+\Bbb R_{\geq 0}^n).$$
where $\text{supp}(f)=\{\alpha| c_\alpha\neq 0\}$ denotes the support of $f$. This is an unbounded polytope in $\Bbb R^n$. We call the union $\Gamma (f)$ of its compact faces the {\em Newton diagram} of $f$. By $\Gamma_-(f)$ we denote the union of all line segments joining the origin to a point
on $\Gamma(f)$. We always assume that $f\in\mathfrak m$ if not explicitly stated otherwise.

If the Newton diagram of a singularity $f$ meets all coordinate axes we call $f$ {\em convenient}. However, not every isolated singularity is convenient, and one then has to enlarge the Newton diagram. A compact rational polytope $P$ of dimension $n-1$ in the positive orthant $\Bbb R^n_{\geq 0}$ is called a {\em $C$-polytope} if the region above $P$ is convex and if every ray in the positive orthant emanating from the origin meets $P$ in exactly one point. The Newton diagram of $f$ is a $C$-polytope iff $f$ is convenient.

We first introduce the different notions of non-degeneracy. For this let $f =\sum_{\alpha}c_\alpha x^\alpha\in \mathfrak m$ be a power series, let $P$ be a $C$-polytope and let $\Delta$ be a face of $P$. By $f_\Delta:=\textrm{in}_\Delta(f) :=\sum_{\alpha\in\Delta} c_\alpha x^\alpha$ we denote the initial form or principal part of f along $\Delta$. Following Kouchnirenko we call $f$ {\em non-degenerate} ND {\em along} $\Delta$ if the Jacobian ideal\footnote{The Jacobian ideal $j(f)$ denotes the ideal generated by all partials of $f\in K[[x]]$.} $\textrm{j}(f_\Delta)$ has no zero in the torus $(K^*)^n$. $f$ is then said to be {\em Newton non-degenerate} NND if $f$ is non-degenerate along each face (of any dimension) of the Newton diagram $\Gamma(f)$. We do not require $f$ to be convenient. 

To define inner non-degeneracy we need to fix two more notions. The face $\Delta$ is an {\em inner face} of $P$ if it is not contained in any coordinate hyperplane. Each point $q\in K^n$ determines a coordinate hyperspace $H_q =\bigcap_{q_i = 0} \{x_i=0\}\subset \Bbb R^n$ in $\Bbb R^n$. We call $f$ {\em inner non-degenerate} IND {\em along} $\Delta$ if for each zero $q$ of the Jacobian ideal $\textrm{j}(\textrm{in}_\Delta(f))$ the polytope $\Delta$ contains no point on $H_q$. $f$ is called {\em inner Newton non-degenerate} INND {\em w.r.t. a $C$-polytope} $P$ if no point of $\text{supp}(f)$ lies below $P$ and $f$ is IND along each inner face of $P$. We call $f$ simply {\em inner Newton non-degenerate} INND if it is INND w.r.t some $C$-polytope.

Finally, we call $f$ {\em weakly non-degenerate} WND {\em along} $\Delta$ if the Tjurina ideal\footnote{For $f\in K[[x]]$ we call $tj(f)=\langle f\rangle+j(f)$ the Tjurina ideal of $f$.} $tj(\textrm{in}_\Delta(f))$ has no zero in the torus $(K^*)^n$, and $f$ is called {\em weakly Newton non-degenerate} WNND if f is weakly non-degenerate along each top-dimensional face of $\Gamma(f)$. Note that NND implies WNND while NND does not imply INND and vice versa. See \cite[Remark 3.1]{BGM10} for facts on and relations between the different types of non-degeneracy.

For any compact polytope $Q$ in $\Bbb R^n_{\geq 0}$ we denote by $V_k(Q)$ the sum of the $k$-dimensional Euclidean volumes of the intersections of $Q$ with the $k$-dimensional coordinate subspaces of $\Bbb R^n$ and, following Kouchnirenko, we then call
$$\mu_N(Q) =\sum_{k=0}^n (-1)^{n-k}  k! V_k(Q)$$
the Newton number of $Q$. For a power series $f \in K[[x]]$ we define the {\em Newton number} of $f$ to be
$$\mu_N(f) = \sup\{ \mu_N(\Gamma_-(f_m)) | f_m:=f+x_1^m+\ldots+ x_n^m, m\geq 1 \}.$$
If $f$ is convenient then $$\mu_N(f) = \mu_N(\Gamma_-(f)).$$
The following theorem was proved by Kouchnirenko in arbitrary characteristic. We recall that $\mu(f):=\dim K[[x,y]]/j(f)$ is the {\em Milnor number} of $f$.
\begin{theorem}\cite{Kou76}\label{thm0}
For $f \in K[[x]]$ we have $\mu_N(f) \leq \mu(f)$, and if $f$ is NND and convenient then
$\mu_N(f) = \mu(f)<\infty$.
\end{theorem}
Since Theorem \ref{thm0} does not cover all semi-quasihomogeneous singularities, Wall introduced the condition INND (denoted by NPND* in \cite{Wal99}). Using Theorem \ref{thm0}, Wall proved the following theorem for $K=\Bbb C$ which was extended to arbitrary $K$ in \cite{BGM10}.
\begin{theorem}\cite{Wal99}, \cite{BGM10} \label{thm1}
If $f \in K[[x]]$ is INND, then
$$\mu(f)=\mu_N(f)=\mu_N(\Gamma_-(f))<\infty.$$
\end{theorem}
Kouchnirenko proved that the condition "convenient" is not necessary in Theorem \ref{thm0} if $char(K)=0$. The authors in \cite{BGM10} show that in the planar case Kouchnirenko's result holds in arbitrary characteristic without the assumption that $f$ is convenient (allowing $\mu(f)=\infty$):
\begin{proposition}\cite[Proposition 4.5]{BGM10}\label{pro}
Suppose that $f \in K[[x, y]]$ is NND, then $\mu_N(f) = \mu(f)$.
\end{proposition}
\section{Milnor number}
In the following we consider only the case of plane curve singularities. The main result of this section says that for $f\in K[[x,y]]$, the condition $\mu(f)=\mu_N(f)<\infty$ is equivalent to $f$ being INND (Theorem \ref{thm11}). In characteristic zero this is also equivalent to $f$ being NND and $\mu_N(f)<\infty$ (Corollary \ref{cor13}). However, in positive characteristic, this is in general not true as the following example shows.
\begin{example}{\rm
$f=x^3+xy+y^3$ in characteristic 3 satisfies $\mu(f)=\mu_N(f)=1$ but $f$ is not NND.
}\end{example}
\begin{remark}\label{rm11}{\rm
Let $f \in K[[x,y]]$ be convenient and $A_i=(c_i,e_i), i=0,\ldots,k$ the vertices of $\Gamma(f)$ with $c_0=e_k=0,c_{i}<c_{i+1}$ and $e_{i}>e_{i+1}$. Then
$$\mu_N(f) =2V_2 (\Gamma_-(f))-c_k-e_0+1.$$
}
\end{remark}

\begin{lemma}\label{lm11}
Let $f,g\in K[[x,y]]$ be convenient such that $\Gamma_-(f)\subset \Gamma_-(g)$. Then
\begin{itemize}
\item[(a)] $\mu_N(f)\leq \mu_N(g)$.
\item[(b)] The equality holds if and only if  $\Gamma_-(f)\cap \Bbb R^2_{\geq 1}= \Gamma_-(g)\cap \Bbb R^2_{\geq 1}$, where 
$$\Bbb R^2_{\geq 1}=\{(x,y)\in\Bbb R^2| x\geq 1, y\geq 1\}.$$
\end{itemize}
\end{lemma}
Part (a) of the lemma was also shown in \cite[Coro. 5.6]{Biv09}. Let us denote by $\Gamma_1(f)$ the cone joining the origin with $\Gamma(f)\cap \Bbb R^2_{\geq 1}$. (cf. Fig. 1). 

\setlength{\unitlength}{0.2cm}
\centerline{\begin{picture}(20,22)(0,-3)
\linethickness{0.05mm}
\multiput(0,0)(2,0){9}%
{\line(0,1){16}}
\multiput(0,0)(0,2){8}%
{\line(1,0){18}}
\put(0,0){\line(1,6){1.8}}
\put(0,0){\line(6,1){12}}
\put(0,14){\line(2,-3){4}}
\put(4,8){\line(1,-1){4}}
\put(8,4){\line(2,-1){8}}
\put(0,0){\circle*{0.8}}
\put(0,14){\circle*{0.8}}
\put(2,11){\circle*{0.8}}
\put(4,8){\circle*{0.8}}
\put(8,4){\circle*{0.8}}
\put(12,2){\circle*{0.8}}
\put(16,0){\circle*{0.8}}
\put(-1.4,-1.4){$0$}
\put(-2,14){$e_0$}
\put(15,-2){$c_k$}
\put(1.5,-1.4){$1$}
\put(-1.4,1.5){$1$}
\linethickness{5mm}
\put(4,8){\line(-1,-1){3.2}}
\put(8,4){\line(-1,-1){3.2}}
\put(6,6){\line(-1,-1){6}}
\put(5,7){\line(-1,-1){4.6}}
\put(3.2,9.2){\line(-1,-1){2}}
\put(7,5){\line(-1,-1){4.6}}
\put(9.3,3.3){\line(-1,-1){2.1}}
\put(-8,8){$\Gamma_1(f)$}
\put(20,6){$\Gamma(f)$}
\put(-4,7){\vector(3,-1){8}}
\put(20,6){\vector(-4,-1){10.6}}
\put(7,-5){Fig. 1.}
\end{picture}}
\begin{proof}
First, we prove that $$\mu_N(f)=V_2(\Gamma_1(f))+1.$$
It is easy to see that $\Gamma_1(f)$ divides $\Gamma_-(f)$ into three parts whose volumes are $c_k/2,V_2(\Gamma_1(f))$ and $e_0/2$. Therefore
$$\mu_N(f) =2V_2 (\Gamma_-(f))-c_k-e_0+1=2V_2 (\Gamma_1(f))+1.$$
(a) Clearly, if $\Gamma_-(f)\subset \Gamma_-(g)$ then $\Gamma_1(f)\subset \Gamma_1(g)$ and hence 
$$\mu_N(f)=V_2(\Gamma_1(f))+1\leq 2V_2(\Gamma_1(g))+1=\mu_N(g).$$
(b) follows easily from the formula $\mu_N(f)=2V_2(\Gamma_1(f))+1.$
\end{proof}
We recall some classical notions. Let $f\in K[[x,y]]$ be irreducible. A couple $(x(t),y(t))\in K[[t]]^2$ is called a {\em (primitive) parametrization} of $f$, if $f(x(t),y(t))=0$ and if the following universal factorization property holds: for each $(u(t),v(t))\in K[[t]]^2$ with $f(u(t),v(t))=0$, there exists a unique series $h(t)\in K[[t]]$ such that $u(t)=x(h(t)) \text{ and } v(t)=y(h(t)).$ 

If $g\in K[[x,y]]$ is irreducible and $(x(t),y(t))$ its parametrization, then the {\em intersection multiplicity} of any $f\in K[[x,y]]$ with $g$ is given by $i(f,g)=\mathrm{ord} f(x(t),y(t))$, and if $u$ is a unit then $i(f,u)=0$.
The {\em intersection multiplicity} of $f$ with a reducible power series $g=g_1\cdot\ldots\cdot g_s$ is defined to be the sum $i(f,g)=i(f,g_1)+\ldots+i(f,g_s).$

\begin{proposition}\cite[Pro. 3.12]{GLS06}\label{pro0}
Let $f,g\in K[[x,y]]$. Then
$$i(f,g)=i(g,f)=\dim K[[x,y]]/\langle f,g \rangle.$$
\end{proposition}
The proof in \cite{GLS06} was given for $K=\Bbb C$ but works in any characteristic.

Let $f=\sum_{i,j}c_{ij}x^{i}y^{j}\in K[[x,y]]$ and $\Gamma(f)$ be its Newton diagram. We call
$$f_{in}:=\sum_{(i,j)\in\Gamma(f)} c_{ij}x^{i}y^{j}$$
the {\em initial part} of $f$.
\begin{proposition}\label{pro1}
Let $f\in\mathfrak m\subset K[[x,y]]$ be irreducible, $x$-general of order $m$ and $y$-general of order $n$. Let $(x(t),y(t))$ be parametrization of $f$. Then
\begin{itemize}
\item[(a)] $\mathrm{ord} (x(t))=n$ and $\mathrm{ord} (y(t))=m$.
\item[(b)] The Newton diagram of $f$ is the straight line segment.
\item[(c)] There exist $\xi, \lambda\in K^*$ such that 
$$f_{in}(x,y)=\xi\cdot (x^{m/q}-\lambda y^{n/q})^q,$$
where $q=(m,n)$.
\end{itemize}
\end{proposition}
\begin{proof}
cf. \cite[Lemma 3.4.3, 3.4.4, 3.4.5]{Cam80}.
\end{proof}

\begin{proposition}\label{pro4}\cite[Lemma 3]{BrK86}
Let $f\in K[[x,y]]$ and let $E_i, i=1,\ldots, k$ be the edges of its Newton diagram. Then there is a factorization of $f$: 
$$f=monomial\cdot \bar f_1\cdot\ldots\cdot \bar f_k$$
such that $\bar f_i$ is convenient, $f_{E_i}=monomial\times (\bar f_i)_{in}$. In particular, if $f$ is convenient then $f=\bar f_1\cdot\ldots\cdot \bar f_k$.
\end{proposition}
A polynomial $f=\sum_{i,j}c_{ij}x^{i}y^{j}\in K[x,y]$ is called {\it weighted homogeneous} or {\it quasihomogeneous} of type $(n,m;d)$ if $m,n,d$ are positive integers satisfying $n i+mj=d$, for each $(i,j)\in\text{supp}(f)$.

Let $f\in K[[x,y]]$ be a formal power series and $n,m$ positive integers. We can decompose $f$ into a sum
$$f=f^{\mathrm{w}}_d+f^{\mathrm{w}}_{d+1}+\ldots,$$
where $f^{\mathrm{w}}_d\neq 0$ and $f^{\mathrm{w}}_l$ is weighted homogeneous of type $(n,m;l)$ for $l\geq d$. We call $f^{\mathrm{w}}_d$ the first term of the decomposition.

For each series $\varphi(t)=c_1t^{\alpha_1}+c_2 t^{\alpha_2}+\ldots$ with $c_1\neq 0, \alpha_1<\alpha_2<\ldots$, we set
$$LT(\varphi(t)):=c_1t^{\alpha_1}\text{ and } LC(\varphi(t)):=c_1.$$
\begin{lemma}\label{lm12} 
Let $m,n$ be two positive integers. Let $x(t),y(t)\in K[[t]]$ with $\mathrm{LT} (x(t))=a t^{\alpha}$ and $\mathrm{LT} (y(t))=b t^{\beta}$ such that $\alpha:\beta=n:m$. Let $f=f^{\mathrm{w}}_d+f^{\mathrm{w}}_{d+1}+\ldots$, be a $(n,m)$-weighted homogeneous decomposition of $f\in K[[x,y]]$. Then $\mathrm{ord}f(x(t),y(t))\geq \frac{d\alpha}{n}$. Equality holds if and only if $f_d(a,b)\neq 0$.
\end{lemma}
\begin{proof}
We can write $x(t)=t^{\alpha}(a+u(t))$ and $y(t)=t^{\beta}(b+v(t))$, where $\mathrm{ord} u(t)>0$ and $\mathrm{ord} v(t)>0$. Then 
\begin{eqnarray*}
f^{\mathrm{w}}_l(x(t),y(t))&=&\sum_{ni+mj=l}c_{ij}(t^{\alpha}(a+u(t)))^i(t^{\beta}(b+v(t)))^j\\
&=& t^\frac{l\alpha}{n}f^{\mathrm{w}}_l(a+u(t),b+v(t)).
\end{eqnarray*}
Thus $\mathrm{ord}f^{\mathrm{w}}_l(x(t),y(t))\geq \frac{l\alpha}{n}$ and hence $\mathrm{ord}f(x(t),y(t))\geq \frac{d\alpha}{n}$.

Since $f^{\mathrm{w}}_d(a+u(t),b+v(t))=f^{\mathrm{w}}_d(a,b)+t h(t)$ for some power series $h$, 
$$\mathrm{ord}f(x(t),y(t))= \mathrm{ord} f^{\mathrm{w}}_d(x(t),y(t))=\frac{d\alpha}{n}$$
iff $f^{\mathrm{w}}_d(a,b)\neq 0$.
\end{proof}
\begin{lemma}\label{lm13} 
Let $f\in K[[x,y]]$ be convenient such that $\Gamma(f)$ has only one edge. Let $m=\mathrm{ord} f(x,0), n=\mathrm{ord} f(0,y)$ and $f= f_1\cdot\ldots\cdot  f_r$ a factorization of $f$ into its branches (irreducible factors).
\begin{itemize}
\item[(a)] Let $(x_j(t),y_j(t))$ be a parametrization of $ f_j, j=1,\ldots,r$ with $\mathrm{LT}(x_j(t))=a_j t^{\alpha_j}$ and $\mathrm{LT}(y_j(t))=b_j t^{\beta_j}$. Then $f_{\mathrm {in}}(a_j,b_j)= 0$, $\alpha_j:\beta_j=n:m$ and $\alpha_1+\cdots+\alpha_r=n$.
\item[(b)] Let $a, b\in K^*$ such that $f_{\mathrm {in}}(a,b)= 0$. Then there is a parametrization $(x(t),y(t))$ of a branch of $f$ satisfying $\mathrm{LC}(x(t))=a$ and $\mathrm{LC}(y(t))=b$. 
\end{itemize}
\end{lemma}
\begin{proof} Let $f=f^{\mathrm{w}}_{d}+f^{\mathrm{w}}_{{d}+1}+\ldots$ with $f^{\mathrm{w}}_{d}\neq 0$ be the $(n,m)$-weighted homogeneous decomposition of $f$. Then $f^{\mathrm{w}}_d=f_{\mathrm{in}}$. 

(a) It is easily verified that $f_{\mathrm {in}}=\prod( f_j)_{\mathrm {in}}$ then $( f_j)_{in}$ is also a $(n,m)$-weighted homogeneous polynomial of order some $d_j$. By Proposition \ref{pro1}, $\mathrm{ord}  f_j(x,0)=\beta_j$ and $\mathrm{ord}  f_j(0,y)=\alpha_j$, i.e. $x^{\beta_j}$ and $y^{\alpha_j}$ are monomials of $( f_j)_{\mathrm {in}}$. Thus $n\beta_j=d_j=m\alpha_j$ and hence $\alpha_j:\beta_j=n:m$.

Since $f(x_j(t),y_j(t))=0$, i.e. $\mathrm{ord} f(x_j(t),y_j(t))=+\infty>\frac{d\alpha_j}{n}$, Lemma \ref{lm12} yields that $f^{\mathrm{w}}_{d}(a_j,b_j)= 0$, i.e. $f_{\mathrm {in}}(a_j,b_j)= 0$. 

Now, by the definition of intersection multiplicity we have
$$n=i(f,x)=\sum_{j=1}^r \mathrm{ord} x_j(t)=\sum_{j=1}^r \alpha_j.$$

(b) Since $f=g\cdot h$ implies $f_{\mathrm {in}}=g_{\mathrm {in}}\cdot h_{\mathrm {in}}$, it suffices to prove part (b) for the irreducible case. Then by Proposition \ref{pro1}(c), there exist $\xi, \lambda\in K^*$ such that 
$$f_{in}(x,y)=\rho \cdot (x^{m'}-\lambda y^{n'})^q,$$
where $q=(m,n); m'=m/q$ and $n'=n/q$. Clearly, $(m',n')=1$. Then it is impossible for the characteristic $p$ of $K$ divides both $m'$ and $n'$. We may assume that $p$ does not divide $n'$. 
 
Let $(\bar x(t),\bar y(t))$ be a parametrization of $f$. It follows from Proposition \ref{pro1} that $\mathrm{LT}(\bar x(t))=\bar a t^n$ and $\mathrm{LT}(\bar y(t))=\bar b t^m$ for some $\bar a, \bar b\in K^*$ satisfying $f_{\mathrm {in}}(\bar a,\bar b)=0$. Set 
$$g(y):= a^{m'}-\lambda y^{n'}\text{ and }\sqrt[n']{a/\bar a}:=\{\xi_i| i=1,\ldots,n'\}.$$
Then 
$$g(\bar b \xi_i^{m'})=f_{\mathrm {in}}(a,\bar b \xi_i^{m'})=f_{\mathrm {in}}(\bar a\xi_i^n,\bar b \xi_i^{m'})=\xi_i^d f_{\mathrm {in}}(\bar a,\bar b)=0.$$
Since $(m',n')=1$, it is easy to see that $\xi_i^{m'}\neq \xi_j^{m'}$ for all $i\neq j$. Thus the set 
$$\{\bar b \xi_i^{m'}| i=1,\ldots,n'\}$$ 
contains all of roots of $g$. Since $0=f_{in}(a,b)=\rho \cdot g(b)^q$, $g(b)=0$. Then there is an index $i_0$ such that $b=\bar b\xi_{i_0}^{m'}$. Choose a $\epsilon$ in $\sqrt[q]{\xi_{i_0}}$ and put
$$x(t)=\bar x(\epsilon t)\text{ and } y(t)=\bar y(\epsilon t),$$
we get $\mathrm{LC}(x(t))=a$ and $\mathrm{LC}(y(t))=b$.
\end{proof}
\begin{definition}\label{def11}{\rm
Let $f=\sum c_{ij}x^iy^j \in K[[x, y]]$ be such that $(0,n)$ is the vertex on the $y$-axis of $\Gamma(f)$. Let $(1,j_1)$ be the intersection point of $\Gamma(f)$ and the line $x=1$. We define $f$ to be {\em ND1 along} $(0,n)$ if either $char(K)=p=0$ or if $p\neq 0$ then $p\not |n$ or $j_1\in\Bbb N$ and the coefficient $c_{1j_1}$ of $xy^{j_1}$ in $f$ is different from zero. ND1 along $(m,0)$, with $(m,0)$ the vertex on the $x$-axis of $\Gamma(f)$, is defined analogously.

$f$ is called NND1 if $f$ is convenient, ND along each inner face and ND1 along each vertex on the axes of $\Gamma(f)$.
}
\end{definition}
\begin{proposition}\label{pro11}
Let $f=\sum c_{ij}x^iy^j \in K[[x, y]]$ be convenient and let $(0,n)$ (resp. $(m,0)$) be the vertex on the $y$-axis (resp. on the $x$-axis) of $\Gamma(f)$. Assume that $f$ is not ND1 along the point $(0,n)$ or $(m,0)$ then $\mu(f)>\mu_N(f)$.
\end{proposition}
\begin{proof}
We consider only $(0,n)$ since $(m,0)$ is analogous. Let $(1,j_1)$ be the intersection point of $\Gamma(f)$ and the line $x=1$. The assumption that $f$ is not ND1 along the point $(0,n)$ implies that $p|n$ and $c_{1j_1}=0$. Putting $g(x,y)=f(x,y)-c_{0n}y^n$ one then has $\mu(f)=\mu(g)$ and $\Gamma_-(f)\subset \Gamma_-(g)$. On the other hand, it is easy to see that $(1,j_1)\in \Gamma_+(f)\setminus \Gamma_+(g)$. This means $\Gamma_-(f)\cap \Bbb R_{\geq 1}^2\subsetneqq \Gamma_-(g)\cap \Bbb R_{\geq 1}^2$. It hence follows from Lemma \ref{lm11} that $\mu_N(g)>\mu_N(f)$. Thus
$$\mu(f)=\mu(g)\geq \mu_N(g)>\mu_N(f).$$
\end{proof}
\begin{proposition}\label{pro12}
Let $f \in K[[x, y]]$ be convenient. If $f$ is degenerate along some inner vertex of $\Gamma(f)$ then $\mu(f)>\mu_N(f)$.
\end{proposition}
\begin{proof}
Assume that $f$ is degenerate along some vertex $(i_0,j_0)$ of $\Gamma(f)$ with $i_0>0$ and $j_0>0$. Then $p\neq 0$ and $i_0$ and $j_0$ are divisible by $p$. Put $g(x,y)=f(x,y)-c_{i_0j_0}x^{i_0}y^{j_0}$, then $j(f)=j(g)$ and hence $\mu(f)=\mu(g)$. Clearly, $\Gamma_+(g)$ does not contain the point $(i_0,j_0)$. Thus 
$$\Gamma_-(f)\cap \Bbb R^2_{\geq 1}\subsetneqq \Gamma_-(g)\cap \Bbb R^2_{\geq 1}.$$ Lemma \ref{lm11} hence implies that $\mu_N(g)>\mu_N(f)$. We then have
$$\mu(f)=\mu(g)\geq \mu_N(g)>\mu_N(f).$$
\end{proof}
\begin{proposition}\label{pro13}
Let $f \in K[[x, y]]$ be convenient. If $f$ is degenerate along some edge of $\Gamma(f)$ then $\mu(f)>\mu_N(f)$.
\end{proposition}
\begin{proof}
Let $f(x,y)=\sum c_{\alpha\beta }x^\alpha y^\beta$. Let $f_x,f_y$ be the partials of $f$ and put $h(x,y):=xf_x(x,y)+\lambda y f_y(x,y),$ where $\lambda\in K$ is generic. Then
$$h(x,y)=\sum(\alpha +\lambda \beta )c_{\alpha \beta }x^\alpha y^\beta .$$
Thus $\text{supp}(h)=\text{supp}(f)\setminus (p\Bbb N)^2$  and if $p=0$ then $\text{supp}(h)=\text{supp}(f)$. Hence $\Gamma_+(h)~\subset \Gamma_+(f)$.

\noindent{\bf Case 1:} $f$ is ND along each vertex of $\Gamma(f)$.

Assume now that $(i,j)$ is a vertex of $\Gamma(f)$. Since $f$ is ND along $(i,j)$, $p=0$ or $p\neq 0$ and one of $i,j$ is not divisible by $p$. Therefore $(i,j)\in \text{supp}(f)\setminus (p\Bbb N)^2=\text{supp}(h)$ and then $\Gamma_+(f)\subset \Gamma_+(h)$. Hence $\Gamma(h)= \Gamma(f)$. 

Let $E_i, i=1,\ldots, k$ be edges of $\Gamma(h)$. By Proposition \ref{pro4}, we can write $h=\bar h_1\ldots \bar h_k$, where $\bar h_i$ are convenient and $h_{E_i}(x,y))=monomial\times (\bar h_i)_{in}$. We denote by $m_i$ and $n_i$ the lengths of the projections of $E_i$ on the horizontal and vertical axes.

Let $h=h^{\mathrm{w}}_{d_i}+h^{\mathrm{w}}_{{d_i}+1}+\ldots$ with $h^{\mathrm{w}}_{d_i}\neq 0$ be the $(n_i,m_i)$-weighted homogeneous decomposition of $h$. Then $h^{\mathrm{w}}_{d_i}=h_{E_i}$. Since $E_i$ is also an edge of $\Gamma(f)$, $f=f^{\mathrm{w}}_{d_i}+f^{\mathrm{w}}_{{d_i}+1}+\ldots$ is the $(n_i,m_i)$-weighted homogeneous decomposition of $f$ with $f^{\mathrm{w}}_{d_i}=f_{E_i}$ and then
$$h^{\mathrm{w}}_{d_i}=\sum_{n_i\alpha+m_i\beta=d_i} (\alpha +\lambda \beta )c_{\alpha \beta }x^\alpha y^\beta=x\frac{\partial f^{\mathrm{w}}_{d_i}}{\partial x}+\lambda y\frac{\partial f^{\mathrm{w}}_{d_i}}{\partial y}.$$
Let $yf_y=g^{\mathrm{w}}_{{d_i}^{'}}+g^{\mathrm{w}}_{{d_i}^{'}+1}+\ldots$ be the $(n_i,m_i)$-weighted homogeneous decomposition of $yf_y$. It is easy to see that ${d_i}^{'}\geq {d_i}$ and ${d_i}^{'}={d_i}$ iff $y\frac{\partial f^{\mathrm{w}}_{d_i}}{\partial y}\neq 0$.
\begin{claim}\label{cl1}
Let $A_{i-1},A_i$ be the vertices of the edge $E_i$ and let $V_2 (OA_{i-1}A_i)$ be the volume of triangle $OA_{i-1}A_i$. Then ${d_i}=2V_2 (OA_{i-1}A_i)$.
\end{claim}
{\em Proof.} Let $(c_i,e_i)$ be the coordinates of $A_i, i=0,\ldots,k$. Then $m_i=c_i-c_{i-1}$ and $n_i=e_{i-1}-e_i$. (cf. Fig. 2). 

\setlength{\unitlength}{0.2cm}
\centerline{\begin{picture}(20,20)(0,-6)
\linethickness{0.05mm}
\multiput(0,0)(2,0){8}%
{\line(0,1){14}}
\multiput(0,0)(0,2){7}%
{\line(1,0){16}}
\linethickness{0.5mm}
\put(4,8){\line(3,-2){6}}
\put(4,8){\circle*{0.8}}
\put(10,4){\circle*{0.8}}
\put(10,8){\circle*{0.6}}
\put(0,8){\circle*{0.6}}
\put(10,0){\circle*{0.6}}
\put(0,0){\circle*{0.8}}
\put(0,0){\line(1,2){4}}
\put(0,0){\line(5,2){10}}
\put(0,0){\line(1,1){6.4}}
\put(2,4){\line(1,1){3.2}}
\put(3.5,1.5){\line(1,1){4.1}}
\put(6.8,2.8){\line(1,1){2}}
\put(-1.4,-1.4){$0$}
\put(-3.5,8){$e_{i-1}$}
\put(10,-1.4){$c_{i}$}
\put(3,-1.4){$c_{i-1}$}
\put(-2,4){$e_{i}$}
\put(10.5,4.5){$A_{i}$}
\put(2.5,9){$A_{i-1}$}
\put(6,-5){Fig. 2.}
\end{picture}}

Considering the rectangle $(0,0);(c_i,0);(c_i,e_{i-1});(0,e_{i-1})$ we have
\begin{eqnarray*}
2V_2(OA_{i-1}A_i)&=&2c_{i}e_{i-1}-c_{i}e_{i}-c_{i-1}e_{i-1}-m_in_i\\
&=& (c_{i-1}+m_i)e_{i-1}+c_{i}(e_{i}+n_i)-c_{i}e_{i}-c_{i-1}e_{i-1}-m_in_i\\
&=& m_ie_{i-1}+c_{i}n_i-m_in_i= m_i (e_{i}+n_i)+c_{i}n_i-m_in_i\\
&=& m_ie_{i}+n_ic_{i}= d_i
\end{eqnarray*}
This proves Claim 1.
\begin{claim}\label{cl2}
$i(\bar h_i,yf_y)\geq {d_i}$, and if $f$ is degenerate along $E_i$ then ~$i(\bar h_i,yf_y)> {d_i}$.
\end{claim}
{\em Proof.} Let $(x_j(t),y_j(t)), j=1,\ldots,r$ be parametrizations of the branches $\bar h_{i,j}$ of $\bar h_i$. Then by Lemma \ref{lm13}, we have
$\text{LT}(x_j(t))=a_j t^{\alpha_j}$ and $\text{LT}(y_j(t))= b_jt^{\beta_j}$, where $a_j,b_j\in K^*, \bar h_i(a_j,b_j)=0, \alpha_j:\beta_j=n_i:m_i$ for all $j=1,\ldots,r$ and $\alpha_1+\ldots+\alpha_r=n_i$. It follows from Lemma \ref{lm12} that $\mathrm{ord}(yf_y)(x_j(t),y_j(t))\geq \frac{{d_i}^{'}\alpha_j}{n_i}$ for all $j=1,\ldots,r$. Thus 
$$i(\bar h_i,yf_y)=\sum_{j=1}^r \mathrm{ord}(yf_y)(x_j(t),y_j(t))\geq\sum_{j=1}^r \frac{{d_i}^{'}\alpha_j}{n_i}= {d_i}^{'}\geq {d_i}.$$

Assume that $f$ is degenerate along $E_i$ then there exist $a,b\neq 0$ such that 
$$x\frac{\partial f^{\mathrm{w}}_{d_i}}{\partial x}(a,b)=y\frac{\partial f^{\mathrm{w}}_{d_i}}{\partial y}(a,b)=0.$$
Therefore $h_{d_i}(a,b)=0$. Lemma \ref{lm13} implies that there is a parametrization of a branch of $\bar h_i$ such that $\text{LT}(\bar x(t))=a t^{\alpha}$ and $\text{LT}(\bar y(t))=b t^{\beta}$. We may assume that $(\bar x(t),\bar y(t))$ is a parametrization of the branch $\bar h_{i,1}$. Then $\alpha=\alpha_1$ and $\beta=\beta_1$.

To show $i(\bar h_i,yf_y)>{d_i}$, we may restrict to the case that ${d_i}^{'}= {d_i}$, because of the inequality $i(\bar h_i,yf_y)\geq {d_i}^{'}\geq {d_i}$. As ${d_i}^{'}={d_i}$ then $g^{\mathrm{w}}_{{d_i}^{'}}(a,b)=y\frac{\partial f^{\mathrm{w}}_{d_i}}{\partial y}(a,b)=0$. Lemma \ref{lm12} yields 
$$\mathrm{ord}(yf_y)(\bar x(t),\bar y(t))>\frac{{d_i}\alpha_1}{n_i}.$$
Thus 
\begin{eqnarray*}
i(\bar h_i,yf_y)&=&\mathrm{ord}(yf_y)(\bar x(t),\bar y(t))+\sum_{j=2}^r \mathrm{ord}(yf_y)(x_j(t),y_j(t))\\
&>&\frac{{d_i}\alpha_1}{n_i}+\sum_{j=2}^r \frac{{d_i}^{'}\alpha_j}{n_i}={d_i}.
\end{eqnarray*}
This proves Claim 2.

It now follows from Claim \ref{cl1} and Claim \ref{cl2} that
$$i(h,yf_y)\geq \sum_{i=1}^{k} 2V_2(OA_{i-1}A_i)=2V_2(\Gamma_-(f)).$$ 
Hence
\begin{eqnarray*}
\mu(f)=i(f_x,f_y)&=&i(h,yf_y)-i(x,f_y)-i(f_x,y)-1\\
&\geq& 2V_2(\Gamma_-(f))-(e_0-1)-(c_k-1)-1\\
&=& \mu_N(f).
\end{eqnarray*}
Moreover, if $f$ is degenerate along some edge of $\Gamma(f)$ then $\mu(f)>\mu_N(f)$ by Claim 1 and 2. This proves of Case 1.

\noindent{\bf Case 2:} In the general case, by propositions \ref{pro11}, \ref{pro12} we may assume that $f$ is ND along each inner vertex and ND1 along the two vertice on the axes of $\Gamma(f)$. For $m$ sufficiently large and $p\not | m$, we put
$$\bar f_m(x,y)=\sum_{(\alpha,\beta)\not\in (p\Bbb N)^2}c_{\alpha\beta}x^\alpha y^\beta+x^m+y^m.$$
Then 
$$\mu(\bar f_m)=\mu(f_m)=\mu(f)\text{ and }\mu_N(\bar f_m)\geq \mu_N(f_m)=\mu_N(f),$$
where the inequality follows from Lemma \ref{lm11}. 
\begin{claim}\label{cl3}
$\bar f_m$ is degenerate along some edge of $\Gamma(\bar f)$.
\end{claim}
{\em Proof.} By the assumption $f$ is degenerate along some edge $E$ of $\Gamma(f)$. If $E$ is also an edge of $\Gamma(\bar f_m)$ then $j(\textrm{in}_E(\bar f_m)=j(\textrm{in}_E(f))$ and hence $\bar f_m$ is degenerate along $E$. If $E$ is not an edge of $\Gamma(\bar f_m)$, then $E$ must meet the axes since $f$ is ND along each inner vertex of $\Gamma(f)$. We may assume that $(0,n)$ is a vertex of $E$. We will show that 
$$\sharp (\text{supp}(\bar f_m)\cap E)\geq 2.$$
Let $(1,j_1)$ be the intersection point of $E$ and the line $x=1$. Since $f$ is ND1 along $(0,n)$, either $(0,n)\in \text{supp}(\bar f_m)\cap E$ or $(1,j_1)\in \text{supp}(\bar f_m)\cap E$, i.e. $\text{supp}(\bar f_m)\cap E\neq \emptyset$ On the other hand, it is easy to see that $\sharp (\text{supp}(\bar f_m)\cap E)\neq 1$ since $f$ is degenerate along the edge $E$. Hence $\sharp (\text{supp}(\bar f_m)\cap E)\geq 2$. Let us denote by $\bar E$ the convex hull of the set $\text{supp}(\bar f_m)\cap E$. Then $\bar E$ is an edge of $\Gamma(\bar f_m)$ and $j(\textrm{in}_{\bar E}(\bar f_m)=j(\textrm{in}_E(f))$. Thus $\bar f_m$ is degenerate along $\bar E$ since $f$ is degenerate along the edge $E$, which proves Claim 3.

Now, by definition, $\bar f_m$ is ND along each vertex of $\Gamma(\bar f_m)$. Since $\bar f_m$ is degenerate along some edge of $\Gamma(\bar f_m)$, applying the first case to $\bar f_m$, we get $\mu(\bar f_m)>\mu_N(\bar f_m)$. Hence
$$\mu(f)=\mu(\bar f_m)>\mu_N(\bar f_m)\geq\mu_N(f).$$
This proves Proposition \ref{pro13}.
\end{proof}
\begin{theorem}\label{thm11}
Let $f \in \mathfrak m\subset K[[x, y]]$ and let $f_m=f+x^m+y^m$. Then the following are equivalent
\begin{itemize}
\item[(i)] $\mu(f)=\mu_N(f)<\infty$.
\item[(ii)] $\mu(f)<\infty$ and $f_m$ is NND1 for some large integer number $m$.
\item[(iii)] $f$ is INND.
\end{itemize}
\end{theorem}
\begin{proof}
$(i)\Rightarrow (ii):$ Since $\mu(f)=\mu_N(f)<\infty$ we have by definition of $\mu_N(f)$
$$\mu(f_m)=\mu(f)=\mu_N(f)=\mu_N(f_m)<\infty.$$
Combining Propositions \ref{pro11}, \ref{pro12} and \ref{pro13} we get the claim.

$(ii)\Rightarrow (iii):$ Assume that $\mu(f)<\infty$ and $f_m$ is NND1. Firstly, it is easy to see that there is an $M\in\Bbb N$ such that $\Gamma(f)\subset\Gamma(f_M)$. It suffices to show $f$ is INND w.r.t. $\Gamma(f_m)$ for all $m>M$. We argue by contradiction. Suppose that it is not true. Then $f$ is not IND along some edge $\Delta$ of $\Gamma(f_m)$ which meets the axes, since $f_m$ is NND1. We may assume that $\Delta$ meets the axes at $(0,n)$. Let $(k,l)$ be the second vertex of $\Delta$. We consider two cases:

$\bullet$ If $l=0$, i.e. $\Gamma(f_m)$ has only one edge $\Delta$. Then $\Delta$ is also a unique edge of $\Gamma(f)$ and $\text{in}_\Delta(f)=\text{in}_\Delta(f_m)$. Since $f$ is not IND along $\Delta$, there exists $(a,b)\in K\setminus \{(0,0)\}$ which is a zero point of $\textrm{j}(\textrm{in}_\Delta(f))$. Beside, since $f_m$ is ND along $\Delta$, either $a=0$ or $b=0$. Assume that $a=0$ and $b\neq 0$. We will show that $f_m$ is not ND1 along $(0,n)$. Firstly, we write $\textrm{in}_\Delta(f_m)=c_{0n}y^n+x\cdot g(x,y)$, then $\dfrac{\partial \textrm{in}_\Delta(f_m)}{\partial y}=ny^{n-1}+x\cdot \dfrac{\partial g}{\partial y}$. Thus $$\dfrac{\partial \textrm{in}_\Delta(f_m)}{\partial y}(0,b)=ny^{n-1}=0\Rightarrow p\neq 0 \text{ and }p|n.$$
We now write $\textrm{in}_\Delta(f_m)=c_{0n}y^n+c_{1j}xy^{j}+x^2\cdot h(x,y)$, then $$\dfrac{\partial \textrm{in}_\Delta(f_m)}{\partial x}=c_{1j}y^{j}+2x\cdot h(x,y)+x^2\cdot\dfrac{\partial h}{\partial x}.$$ Since $\dfrac{\partial \textrm{in}_\Delta(f_m)}{\partial x}(0,b)=0$, $c_{1j}=0$. Hence $f_m$ is not ND1 along $(0,n)$, a contradiction.

$\bullet$ Assume that $l>0$. If $\Delta$ is also an edge of $\Gamma(f)$ then $\text{in}_\Delta(f)=\text{in}_\Delta(f_m)$. Since $f$ is not IND along $\Delta$, there exists $(a,b)\in K\times K^*$ being a zero of $\textrm{j}(\textrm{in}_\Delta(f))$. Since $f_m$ is ND along $\Delta$, $a=0$. Analogously as above $f_m$ is not ND1 along $(0,n)$ and we get a contradiction. Assume now that $\Delta$ is not an edge of $\Gamma(f)$, i.e. $m=n$ and $x|f(x,y)$. Let $P$ be the end point of $\Gamma(f)$ closest to $y$-axis. It follows from $\Gamma(f)\subset\Gamma(f_M)$ and $m>M$ that $P$ must be a vertex of $\Delta$, i.e. $P=(k,l)$. This implies $f=x^k\cdot h(x,y)$. Since $\mu(f)<\infty$, $k=1$. Then $\textrm{in}_\Delta(f)=c_{0n}y^n+c_{1l}xy^{l}$ and clearly $f$ is always IND along $\Delta$, a contradiction. Hence $f$ is INND w.r.t. $\Gamma(f_m)$ and then it is INND.

$(iii)\Rightarrow (i):$ See Theorem \ref{thm1}.
\end{proof}
\begin{corollary}\label{cor11}
Let $f \in K[[x, y]]$ and let $M\in\Bbb N$ such that $\Gamma(f)\subset\Gamma(f_M)$. Then $f$ is INND if and only if it is INND w.r.t. $\Gamma (f_m)$ for some (equivalently for all) $m>M$.
\end{corollary}
\begin{proof}
One direction is obvious, it remains to show $f$ is INND $\Rightarrow $ $f$ is INND w.r.t. $\Gamma (f_m)$ for all $m>M$. We take $m_1>M$ satisfying Theorem \ref{thm11} and then 
$$f \text{ is INND }\Rightarrow \mu(f)<\infty \text{ and } f_{m_1} \text{ is NND1 }\Rightarrow f \text{ is INND w.r.t. }\Gamma (f_{m_1}).$$
For each inner face $\Delta_m$ of $\Gamma (f_m)$, since $m,m_1>M$, there is an inner face $\Delta_{m_1}$ of $\Gamma (f_{m_1})$ such that $\textrm{in}_{\Delta_{m}}(f)=\textrm{in}_{\Delta_{m_1}}(f)$. Thus $f$ is IND along $\Delta_{m}$ since it is IND along $\Delta_{m_1}$. Hence $f$ is INND w.r.t. $\Gamma (f_m)$.
 \end{proof}
\begin{corollary}\label{cor10}
Let $M\in\Bbb N$ be such that $\Gamma(f)\subset\Gamma(f_M)$. Then Theorem \ref{thm11} holds for each $m>M$. 
\end{corollary}
\begin{remark}{\rm
Let $\mu(f)<\infty$. Then $M$ can be chosen as the maximum of $n_1$ and $m_1$, where $n_1=n$ if $\Gamma(f)\cap \{x=0\}=\{(0,n)\}$ and $n_1=2i_1$ if $\Gamma(f)\cap \{x=0\}=\emptyset$ and $\Gamma(f)\cap \{x=1\}=\{(1,i_1)\}$. Similarly we define $m_1$ with $x$ replaced by $y$. This remark and the previous corollaries are important for concrete computation.
}\end{remark}
\begin{proof}[Proof of Corollary \ref{cor10}]
Clearly, the equivalence $(i)\Leftrightarrow (iii)$ does not depend on $m$ and as in the proof of Theorem \ref{thm11} the implication $(ii)\Rightarrow (iii)$ holds for all $m>M$. It remains to show that $f$ is INND $\Rightarrow $ $f_m$ is NND1. By Corollary \ref{cor11}, it suffices to show that $f$ is INND w.r.t. $\Gamma (f_m)$ $\Rightarrow $ $f_m$ is NND1. By contradiction, suppose that $f$ is INND w.r.t. $\Gamma (f_m)$ and $f_m$ is not NND1. Then $f$ is not ND1 along some vertex of $\Gamma(f_m)$ in the axes. Assume that $f$ is not ND1 along $(0,n)\in\Gamma(f_m)$. Then 
$$p\neq 0, p|n\text{ and } \Gamma(f_m)\cap\{x=1\}\cap\text{supp}(f_m)=\emptyset,$$
i.e. $(f_m)_{in}=c_{0n}y^n+x^2\cdot h(x,y)$. This implies $\mu((f_m)_{in})=\infty$. By Theorem \ref{thm1}, $(f_m)_{in}$ is not INND and then $f_m$ is also not INND, a contradiction.
\end{proof}
\begin{corollary}\label{cor13}
Let $K$ is a field of characteristic zero and $f \in \mathfrak{m}\subset K[[x, y]]$. Then the following are equivalent
\begin{itemize}
\item[(i)] $\mu(f)=\mu_N(f)<\infty$.
\item[(ii)] $f$ is INND.
\item[(iii)] $f$ is NND and $\mu_N(f)<\infty$.
\end{itemize}
In particular, if $f$ is convenient then (i)-(iii) are equivalent to 
\begin{itemize}
\item[(iv)] $f$ is NND.
\end{itemize}
\end{corollary}
\begin{proof}
The implications $(i)\Rightarrow (ii)$ and $(iii)\Rightarrow (i)$ follow from Theorem \ref{thm11} and Proposition \ref{pro}. It remains to prove $(ii)\Rightarrow (iii)$.
 
Assume that $f$ is INND. Then by Theorem \ref{thm11}, $\mu_N(f)<\infty$. We will show that $f$ is ND along each vertex and each edge of $\Gamma(f)$. Since $char (K)=0$, $f$ is ND along each vertex of $\Gamma(f)$. Let $\Delta$ be an edge of $\Gamma(f)$. Clearly, it is an inner edge of $\Gamma(f_m)$, where $m$ sufficiently large. Since $f$ is INND, by Corollary \ref{cor11} $f$ is INND w.r.t. $\Gamma(f_m)$. Then $f$ is IND along $\Delta$, and hence it is also ND along $\Delta$. This implies $f$ is NND. 
\end{proof}
\begin{corollary}
If $f$ is NND and $\mu_N(f)<\infty$ then $f$ is INND.
\end{corollary}
\begin{proof}
This follows from Proposition \ref{pro} and Theorem \ref{thm11}.
\end{proof}
Note that $char(K)=0$ is only used to assure that $f$ is ND along each vertex of $\Gamma(f)\cap (\{0\}\times\Bbb N\cup \Bbb N\times\{0\})$. Hence, the last corollary holds also if $p>0$ and $p\not | n$ if $(0,n)=\Gamma(f)\cap \{0\}\times\Bbb N$ and $p\not | m$ if $(m,0)=\Gamma(f)\cap \Bbb N\times\{0\}$. Example 2.1 shows that this condition is necessary.

\section{$\delta$-Invariant}
We consider now another important invariant of plane curve singularities, the invariant $\delta$ and its combinatorial counterpart, the Newton invariant $\delta_N$. We show that both coincide iff $f$ is weighted homogeneous Newton non-degenerate (WHNND), a new non-degenerate condition introduced below.

Let $f \in\mathfrak m\subset K[[x, y]]$ be a power series. We define {\em the multiplicity} of $f$, denoted by $\mathrm{mt}(f)$, to be the minimal degree of the homogeneous part of $f$. So
$$f=\sum_{k \geq m:=\mathrm{mt}(f)} f_k(x,y),$$
where $f_k$ is homogeneous of degree $k$ and $f_m\neq 0$. Then $f_m$ decomposes into linear factors,
$$f_m=\prod_{i=1}^s (\alpha_i x-\beta_i y)^{r_i},$$
with $(\beta_i:\alpha_i)\in \mathbb P^1$ pairwise distinct. We call $f_m$ {\em the tangent cone} and the points $(\beta_i:\alpha_i), i=1,\ldots,s$, {\em the tangent directions} of $f$.   

We fix a minimal resolution of the singularity computed via successively blowing up points, denote by $Q \to 0$ that $Q$ is an infinitely near point of the origin on $f$. If $Q$ is an infinitely near point in the $n$-th neighbourhood of $0$, we denote by $m_Q$ the multiplicity of the $n$-th strict transform of $f$ at $Q$. If $P$ is  an infinitely near point in the $l$-th neighbourhood of $0$, we denote by $Q\to P$ that $Q$ is also an infinitely near point of $P$ on  the $l$-th strict transform $\tilde f_l$ of $f$ at $P$. Note that if $Q\to P$ then $n\geq l$ and we set $n(\tilde f_l,Q):=n-l$. In particular, we have  $n(f,Q)=n$.

Let $E_1,\ldots, E_k$ be the edges of the Newton diagram of $f$. We denote by $l(E_i)$ the lattice length of $E_i$, i.e. the number of lattice points on $E_i$ minus one and by $s(f_{E_i})$ the number of non-monomial irreducible (reduced) factors of $f_{E_i}$. We set

(a) $\delta(f) :=\sum_{Q\to 0} \frac{m_Q(m_Q-1)}{2}$ the {\em delta invariant} of $f$. The delta invariant $\delta(f)$ equals also $\dim_K (\bar R/R)$ where $R=K[[x,y]]/\langle f \rangle$ and $\bar R$ is the integral closure of of $R$ in its total ring of fractions. 

(b) $\nu(f) :=\sum_{Q\text{ special}} \frac{m_Q(m_Q-1)}{2}$, where an infinitely near point $Q$ is {\em special} if it is the origin or the origin of the corresponding chart of the blowing up.

(c) $r(f)$ the {\em number of branches} of $f$ counted with multiplicity.

(d) If $f$ is convenient, we define
$$\delta_N(f) := V_2(\Gamma_-(f)) -\frac{V_1(\Gamma_-(f))}{2}+\frac{\sum_{i=1}^k l(E_i)}{2},$$
and otherwise we set $\delta_N(f) := \sup\{\delta_N(f^{(m)})| f^{(m)}:=f+x^m+y^m,m \in\Bbb N\}$ and call it the {\em Newton $\delta$-invariant} of $f$. 

(e) $r_N(f):=\sum_{i=1}^k l(E_i)+\max\{j|x^j \text{ divides }f\}+\max\{l|y^l \text{ divides }f\}$.

(f) $s_N(f):=\sum_{i=1}^k s(f_{E_i})+\max\{j|x^j \text{ divides }f\}+\max\{l|y^l \text{ divides }f\}$.

Note that $\delta(f)$ and $r(f)$ are coordinate-independent while all the other ones depend (only) on the Newton diagram of $f$ and hence are coordinate-dependent (for $\nu(f)$ see Proposition \ref{pro21}).
\begin{proposition}\label{pro20}
For $0\neq f\in\langle x,y\rangle$ we have $r(f)\leq r_N(f)$, and if $f$ is WNND then $r(f)= r_N(f)$.
\end{proposition}
\begin{proof}
cf. \cite[Lemma 4.10]{BGM10}
\end{proof}
Let $E$ be an edge of the Newton diagram of $f$. Then we can write $f_E$ as follows,
$$f_{E}=monomial\times\prod_{i=1}^{s}(a_ix^{m_0}-b_iy^{n_0})^{r_i},$$ 
where $a_i,b_i\in K^*$, $(a_i:b_i)$ pairwise distinct; $m_0,n_0,r_i\in \Bbb N_{>0}$, $\text{gcd}(m_0,n_0)=1$. It easy to see that 
$$s=s(f_{E})\text{ and } l(E)=\sum_{i=1}^s r_i.$$ 
This implies $s(f_{E})\leq l(E)$ and hence $s_N(f)\leq r_N(f)$.

Let $f=f^{\mathrm{w}}_{d}+f^{\mathrm{w}}_{{d}+1}+\ldots$ with $f^{\mathrm{w}}_{d}\neq 0$ be the $(n_0,m_0)$-weighted homogeneous decomposition of $f$. 
\begin{definition}{\rm
We say that $f$ is {\em weighted homogeneous non-degenerate} (WHND) {\em along $E$} if either $r_i=1$ for all $i=1,\ldots, s$ or $(a_ix^{m_0}-b_iy^{n_0})$ does not divide $f^{\mathrm{w}}_{d+1}$ for each $r_i>1$. 

$f$ is called {\em weighted homogeneous Newton non-degenerate} (WHNND) if its Newton diagram has no edge or if it is WHND along each edge of its Newton diagram.  
}\end{definition}
\begin{lemma}\label{lm20}
Let $f\in K[[x,y]]$. Then $f$ is not WHNND if and only if there exist $a,b\in K^*, m,n\in \Bbb N_{>0}$ with $(m,n)=1$ such that $f^{\mathrm{w}}_{d}$ is divisible by $(ax^m-by^n)^2$ and $f^{\mathrm{w}}_{d+1}$ is divisible by $(ax^m-by^n)$, where $f^{\mathrm{w}}_{d}$ (resp. $f^{\mathrm{w}}_{d+1}$) is the first (resp. the second) term of the $(n,m)$-weighted decomposition of $f$.
\end{lemma}
\begin{proof}
Straightforward from the above definition.
\end{proof}
\begin{remark}{\rm
(a) In \cite{Lu87} the author introduced superisolated singularities to study the $\mu$-constant stratum. We recall that $f\in K[[x,y]]$ is {\em superisolated} if it becomes regular after only one blowing up. By (\cite[Lemma 1]{Lu87}), this is equivalent to: $f_{m+1}(\beta_i,\alpha_i)\neq 0$ for all tangent directions $(\beta_i:\alpha_i)$ of $f$ with $r_i>1$, where $f=f_m+f_{m+1}+\ldots$ is the homogeneous decomposition of $f$ and 
$$f_m=\prod_{i=1}^s (\alpha_i x-\beta_i y)^{r_i}.$$
Note that this condition concerns all factors of $f_m$ including monomials. For WHNND singularities we require a similar condition, but for "all weights" and without any condition on the monomial factors of the first term of the weigted homogeneous decomposition of $f$.

(b) Since a plane curve singularity is superisolated iff it becomes regular after only one blowing up, we have $\delta(f)=\nu(f)=m(m-1)/2$ and hence $\delta(f)=\delta_N(f)=m(m-1)/2$, by Proposition \ref{pro21}. It follows from Theorem \ref{thm21} that 

(c) A superisolated plane curve singularity is WHNND.

(d) The plane curve singularity $x^2+y^5$ is WHNND but not superisolated.
}\end{remark}
\begin{proposition}\label{pro23} With notations as above, $f$ is WND along $E$ if and only if $s(f_E)=l(E)$ or, equivalently, iff $r_{i}=1$ for all $i=1,\ldots,s$. In particular, WNND implies WHNND.
\end{proposition}
\begin{proof}
Firstly we can see that the equation $s(f_E)=l(E)$ is equivalent to $r_{i}=1$ for all $i=1,\ldots,s$ since $s(f_E)=s$ and $l(E)=r_1+\ldots+r_s$. It remains to prove that $f$ is WND along $E$ iff $r_{i}=1$ for all $i=1,\ldots,s$. Assume that there is an $i_0$ s.t. $r_{i_0}>1$. It is easy to see that $f_E, \dfrac{\partial  f_{E}}{\partial x}, \dfrac{\partial  f_{E}}{\partial y}$ are divisible by ($a_{i_0}x_0^{m_0}-b_{i_0}y_0^{n_0}$). Hence $f$ is weakly degenerate (WD) along $E$.

We now assume that $f$ is weakly degenerate (WD) along $E$. Then there exist $x_0,y_0\in K^*$ such that
$$f_{E}(x_0,y_0)=\frac{\partial  f_{E}}{\partial x}(x_0,y_0)=\frac{f_{E}}{\partial y}(x_0,y_0)=0,$$
and hence there exists an index $i_0$ such that $a_{i_0}x_0^{m_0}-b_{i_0}y_0^{n_0}=0$. We will show that $r_{i_0}>1$. In fact, if this is not true then $f_E(x,y)=(a_{i_0}x^{m_0}-b_{i_0}y^{n_0})\cdot h(x,y)$ with $h(x_0,y_0)\neq 0$. Since
$$\frac{\partial  f_{E}}{\partial x}(x_0,y_0)=\frac{f_{E}}{\partial y}(x_0,y_0)=0,$$
this is impossible if $p=0$ and implies that $p$ divides $m_0$ and $n_0$ if $p>0$. This contradicts the assumption $\text{gcd}(m_0,n_0)=1$.
\end{proof}
Let $f\in K[[x,y]]$ and let $E_i, i=1,\ldots, k$ be the edges of its Newton diagram. Then by Proposition \ref{pro4} there is a factorization of $f$,
$$f=monomial\cdot \bar f_1\cdot\ldots\cdot \bar f_k,$$
such that $\bar f_i$ is convenient and $f_{E_i}=monomial\times (\bar f_i)_{in}$. Note that $\bar f_i$ is in general not irreducible. On the other hand, $f$ can be factorized into its irreducible factors as $f=m_1\cdot\ldots\cdot m_l\cdot f_1\cdot \ldots\cdot f_r$, where $m_j$ are monomials, and $f_j$ are convenient.
\begin{proposition}\label{pro26}
\begin{itemize}
\item[(a)] Let $g,h\in K[[x,y]]$ such that $f=g\cdot h$. If $f$ is WHNND then $g$ and $h$ are also WHNND.
\item[(b)] With the above notations, the following are equivalent:
\begin{itemize}
\item[(i)] $f$ is WHNND.
\item[(ii)] $\bar f_1,\ldots, \bar f_k$ are WHNND.
\item[(iii)] $f_1,\ldots, f_r$ are WHNND and $(f_i)_{in}$ are pairwise coprime.
\end{itemize}
\end{itemize}
\end{proposition}
\begin{proof}
(a) It suffice to show that if $g$ is not WHNND then neither is $f$. In fact, since $g$ is not WHNND, by Lemma \ref{lm20}, there exist $a,b\in K^*, m,n\in \Bbb N_{>0}$ with $(m,n)=1$ such that $g^{\mathrm{w}}_{c}$ is divisible by $(ax^m-by^n)^2$ and $g^{\mathrm{w}}_{c+1}$ is divisible by $(ax^m-by^n)$, where $g^{\mathrm{w}}_{c}$ (resp. $g^{\mathrm{w}}_{c+1}$) is the first (resp. the second) term of the $(n,m)$-weighted decomposition of $g$. Let $f=f^{\mathrm{w}}_{d}+f^{\mathrm{w}}_{d+1}+\ldots$ (resp. $h=h^{\mathrm{w}}_{e}+h^{\mathrm{w}}_{e+1}+\ldots$) be the $(n,m)$-weighted homogeneous decomposition of $f$ (resp. $h$). Then 
$$f^{\mathrm{w}}_{d}=g^{\mathrm{w}}_{c}\cdot h^{\mathrm{w}}_{e}\text{ and }f^{\mathrm{w}}_{d+1}=g^{\mathrm{w}}_{c}\cdot h^{\mathrm{w}}_{e+1}+g^{\mathrm{w}}_{c+1}\cdot h^{\mathrm{w}}_{e}.$$
This implies that $f^{\mathrm{w}}_{d}$ is divisible by $(ax^m-by^n)^2$ and $f^{\mathrm{w}}_{d+1}$ is divisible by $(ax^m-by^n)$. Again by Lemma \ref{lm20}, $f$ is not WHNND.
\\
(b) It is easily verified that we may restrict to the case that $f$ is convenient. \\
The implication (i)$\Rightarrow$ (ii) follows from part (a).

(ii)$\Rightarrow$ (iii): Assume that $\bar f_1,\ldots, \bar f_k$ are WHNND. By part (a) we can deduce that $f_1,\ldots, f_r$ are WHNND since for each $i$, $f_i$ is an irreducible factor of some $\bar f_j$. We now show that the $(f_i)_{in}$ are pairwise coprime. By contradiction, suppose that $(f_1)_{in}$ and $(f_2)_{in}$ are not coprime. It follows from Proposition \ref{pro1} that there exist $a,b\in K^*, m,n\in \Bbb N_{>0}$ with $(m,n)=1$ such that $(ax^m-by^n)$ is the unique irreducible factor of $(f_1)_{in}$ and $(f_2)_{in}$. Consequently, $(f_1)_{in}$ and $(f_2)_{in}$ are both $(n,m)$-weighted homogeneous. Assume that $f_1$ resp. $f_2$ is an irreducible factor of $\bar f_{j_1}$ resp. $\bar f_{j_2}$ for some $j_1$ and $j_2$. Since $(\bar f_{j_1})_{in}$ and $(\bar f_{j_2})_{in}$ are weighted homogeneous, $(f_1)_{in}$ resp. $(f_2)_{in}$ is a factor of $(\bar f_{j_1})_{in}$ resp. $(\bar f_{j_2})_{in}$. This implies that $(\bar f_{j_1})_{in}$ and $(\bar f_{j_2})_{in}$ and therefore $f_{E_{j_1}}$ and $f_{E_{j_2}}$ are all $(n,m)$-weighted homogeneous. Then the edge $E_{j_1}$ must coincide the edge $E_{j_2}$ and hence $\bar f_{j_1}=\bar f_{j_2}$. It yields that the product $g:=f_1\cdot f_2$ is a factor of $\bar f_{j_1}$. Now, we decompose $g,f_1,f_2$ into their $(n,m)$-weighted homogeneous terms as follows:
$$g=g^{\mathrm{w}}_{c}+g^{\mathrm{w}}_{c+1}+\ldots, f_1=(f_1)^{\mathrm{w}}_{d_1}+(f_1)^{\mathrm{w}}_{d_1+1}+\ldots, f_2=(f_2)^{\mathrm{w}}_{d_2}+(f_2)^{\mathrm{w}}_{d_2+1}+\ldots.$$
Then $c=d_1+d_2$, $(f_1)^{\mathrm{w}}_{d_1}=(f_1)_{in}$, $(f_2)^{\mathrm{w}}_{d_2}=(f_2)_{in}$, $g^{\mathrm{w}}_{c}=(f_1)^{\mathrm{w}}_{d_1}\cdot (f_2)^{\mathrm{w}}_{d_2}$ and $g^{\mathrm{w}}_{c+1}=(f_1)^{\mathrm{w}}_{d_1}\cdot (f_2)^{\mathrm{w}}_{d_2+1}+(f_1)^{\mathrm{w}}_{d_1+1}\cdot (f_2)^{\mathrm{w}}_{d_2}.$
This implies that $g^{\mathrm{w}}_{c}$ is divisible by $(ax^m-by^n)^2$ and $g^{\mathrm{w}}_{c+1}$ is divisible by $(ax^m-by^n)$. It follows from Lemma \ref{lm20} that $g$ is not WHNND and hence $\bar f_{j_1}$ is also not WHNND by part (a) with $g$ a factor of $\bar f_{j_1}$, which is a contradiction. 

(iii)$\Rightarrow$ (i): Suppose that $f$ is not WHNND and that the $(f_i)_{in}$ are pairwise coprime. We will show that $f_i$ is not WHNND for some $i$. Indeed, since $f$ is not WHNND, by Lemma \ref{lm20}, there exist $a,b\in K^*, m,n\in \Bbb N_{>0}$ with $(m,n)=1$ such that $f^{\mathrm{w}}_{d}$ is divisible by $(ax^m-by^n)^2$ and $f^{\mathrm{w}}_{d+1}$ is divisible by $(ax^m-by^n)$, where $f^{\mathrm{w}}_{d}$ (resp. $f^{\mathrm{w}}_{d+1}$) is the first (resp. the second) term of the $(n,m)$-weighted decomposition of $f$. Let $(f_i)=(f_i)^{\mathrm{w}}_{d_i}+(f_i)^{\mathrm{w}}_{d_i+1}+\ldots$ be the $(n,m)$-weighted homogeneous decomposition of $f_i, i=1,\ldots,r$. Then we have
$$d=\sum_{i=1}^r d_i;\quad (f_i)^{\mathrm{w}}_{d_i}=(f_i)_{in};\quad f^{\mathrm{w}}_{d}=\prod_{i=1}^r (f_i)^{\mathrm{w}}_{d_i};\quad f^{\mathrm{w}}_{d+1}=\sum_{i=1}^r \Big((f_i)^{\mathrm{w}}_{d_i+1}\cdot \prod_{l\neq i} (f_l)^{\mathrm{w}}_{d_l}\Big).$$
Since 
$f^{\mathrm{w}}_{d}=\prod_{i=1}^r (f_i)^{\mathrm{w}}_{d_i}$ and since the $(f_i)^{\mathrm{w}}_{d_i}$ are pairwise coprime, there exists an $i_0$ such that 
$(f_{i_0})^{\mathrm{w}}_{d_{i_0}}$ is divisible by $(ax^m-by^n)^2$ and $(f_l)^{\mathrm{w}}_{d_l}$ is not divisible by $(ax^m-by^n)$ for all $l\neq i_0$. This implies that $(f_{i_0})^{\mathrm{w}}_{d_{i_0}+1}$ is divisible by $(ax^m-by^n)$ since
$$f^{\mathrm{w}}_{d+1}=(f_{i_0})^{\mathrm{w}}_{d_{i_0}+1}\cdot \prod_{l\neq i_0} (f_l)^{\mathrm{w}}_{d_l} +\sum_{i\neq i_0}\Big((f_i)^{\mathrm{w}}_{d_i+1}\cdot \prod_{l\neq i} (f_l)^{\mathrm{w}}_{d_l}\Big).$$
Then $f_{i_0}$ is not WHNND by Lemma \ref{lm20}.
\end{proof}
\begin{proposition}\label{thm20}
For $0\neq f\in\langle x,y\rangle$ we have $s_N(f)\leq r(f)$ and if $f$ is WHNND then $s_N(f)=r(f)$.
\end{proposition}
\begin{proof}
If $f=x^jy^l\cdot g(x,y)$ with $g$ convenient, then
$$s_N(f)=s_N(g)+j+l \text{ and }r(f)=r(g)+j+l,$$ 
so we may assume that $f$ is convenient. \\
{\em Step 1.} Assume first that the Newton diagram $\Gamma(f)$ has only one edge $E$. Then we can see that $f_{in}=\prod_{i=1}^r (f_i)_{in}.$
It follows from Proposition \ref{pro1} that for each $i$, $(f_i)_{in}$ has only one irreducible factor and therefore $f_{in}$ has at most $r$ irreducible factors. This means that $r\geq s_N(f)$.

If $r(f)>s_N(f)$, then there exist $i\neq j$ such that $(f_i)_{in}$ and $(f_j)_{in}$ have the same factor. This means that $(f_i)_{in}$ and $(f_j)_{in}$ are not coprime. Then by Proposition \ref{pro26}, $f$ is not WHNND.
\\
{\em Step 2.} Assume now that the Newton diagram $\Gamma(f)$ has $k$ edges $E_1,\ldots,E_k$. By Proposition \ref{pro4}, $f$ can be factorized as $f=\bar f_1\cdot \ldots\cdot \bar f_k$, where $\bar f_j$ is convenient, its Newton diagram has only one edge and $f_{E_j}=monomial\cdot (\bar f_j)_{in}$ for each $j=1,\ldots,k$. This implies that $s_N(\bar f_j)=s(f_{E_j})$. Then we obtain
$$r(f)=\sum_{j=1}^k r(\bar f_j)\geq \sum_{j=1}^k s_N(\bar f_j)=\sum_{j=1}^k s(f_{E_j})=s_N(f).$$

Now we assume that $r(f)>s_N(f)$. Then there exists a $j=1,\ldots,k$ such that $r(\bar f_{j})>s(f_{E_{j}})=s_N(\bar f_{j})$. It follows from Step 1 that $\bar f_{j}$ is not WHNND. Hence $f$ is not WHNND by Proposition \ref{pro26}, which proves the proposition.
\end{proof}
\begin{proposition}
For $0\neq f\in\langle x,y\rangle$ we have $s_N(f)\leq r(f)\leq r_N(f)$, and both equalities hold if and only if $f$ is WNND.
\end{proposition}
\begin{proof}
The inequalities follow from Proposition \ref{pro20} and Proposition \ref{thm20}. For each edge $E$ of $\Gamma(f)$, by Proposition \ref{pro23},
$f$ is WND along $E$ iff $s(f_E)=l(E).$ This implies that $f$ is WNND if and only if $s_N(f)=r_N(f)$ since $s(f_E)\leq l(E)$ and both sides are additive with respect to edges of $\Gamma(f)$.
\end{proof}
We investigate now the relations between $\nu(f)$, $\delta_N(f)$ and $\delta(f)$, which were studied in \cite{BeP00} and \cite{BGM10}. 
\begin{proposition}\label{pro21}\cite[Lemma 4.8]{BGM10} If $f\in K[[x,y]]$ then $\delta_N(f)=\nu(f)$.
\end{proposition}
\begin{proposition}\label{pro22}\cite[Prop. 4.9]{BGM10} For $0\neq f\in \langle x,y\rangle$ we have $\delta_N(f)\leq\delta(f)$, and if $f$ is WNND then $\delta_N(f)=\delta(f)$.
\end{proposition}
Hence WNND is sufficient but, by the following example, not necessary for $\delta_N(f)=\delta(f)$.
\begin{example}{\rm
Let $f(x,y)=(x+y)^2+y^3\in K[[x,y]]$. Then $f$ is not WNND but $\delta_N(f)=\delta(f)=1$. This easy example shows also that WNND depends on the coordinates since $x^2+y^3$ is WNND. Note that $f$ is WHNND.
}\end{example}
Now we prove that WHNND is necessary and sufficient for $\delta_N(f)=\delta(f)$.
\begin{theorem}\label{thm21}
Let $f \in K[[x, y]]$ be reduced. Then $\delta(f)=\delta_N(f)$ if and only if $f$ is WHNND.
\end{theorem}
We will prove the theorem after three technical lemmas.
 
Let $E$ be an edge of the Newton diagram of $f$. We write
$$f_{E}=monomial\times\prod_{i=1}^{s}(a_ix^{m_0}-b_iy^{n_0})^{r_i},$$ 
where $a_i,b_i\in K^*$, $(a_i:b_i)$ pairwise distinct; $m_0,n_0,r_i\in \Bbb N_{>0}$, $(m_0,n_0)=1$. 
\begin{lemma}\label{lm21}With the above notations, there exist an integer $n$ and an infinitely near point $P_n$ in the $n$-th neighbourhood of $0$, such that
$$(\tilde f_{n})_{E_{n}}(u,v)=monomial\times\prod_{i=1}^{s}(a_{i}u-b_{i}v)^{r_{i}},$$
where $\tilde f_{n}$ is a local equation of the strict transform of $\tilde f$ at $P_n$ and $E_{n}$ is some edge of its Newton diagram $\Gamma(\tilde f_{n})$. Moreover, $f$ is WHND along $E$ if and only if $\tilde f_{n}$ is WHND along $E_{n}$.
\end{lemma}
\begin{proof}
We prove the lemma by induction on $m_0+n_0$. If $m_0+n_0=2$, i.e. $m_0=n_0=1$, then the claim is trivial. 

Suppose $m_0+n_0>2$. Now we show the induction step. Since $m_0+n_0>2$ and $\text{gcd}(m_0,n_0)=1$, $m_0\neq n_0$. We may then assume that $m_0<n_0$. Then $P_1:=(1,0)$ is a special infinitely near point of $0$ and the local equation of $\tilde f_1$ at $P_1$ in chart 2, is:
$$\tilde f_1(x_1,y_1)=\frac{f(x_1y_1,y_1)}{y_1^{m}}, \text{ where } m=\mathrm{mt}(f).$$ 
Let $f=f^{\mathrm{w}}_{d_0}+f^{\mathrm{w}}_{{d_0}+1}+\ldots$ with $f^{\mathrm{w}}_{d_0}\neq 0$ be the $(n_0,m_0)$-weighted homogeneous decomposition of $f$. It is easy to see that $\tilde f_1=(\tilde f_1)^{\mathrm{w}}_{e_0}+(\tilde f_1)^{\mathrm{w}}_{{e_0}+1}+\ldots$ is the $(n_0-m_0,m_0)$-weighted homogeneous decomposition of $\tilde f_1$ with 
$$e_0=d_0-m\cdot m_0\text{ and } (\tilde f_1)^{\mathrm{w}}_{e_0+\nu}=\frac{f^{\mathrm{w}}_{d_0+\nu}(x_1y_1,y_1)}{y_1^{m}}, \forall \nu\geq 0.$$ 
In particular, 
$$(\tilde f_1)^{\mathrm{w}}_{e_0}=monomial\times\prod_{i=1}^{s}(a_{i}x^{m_0}-b_{i}y^{n_0-m_0})^{r_{i}}.$$
We denote by $E_1$ the convex hull of the support of $(\tilde f_1)^{\mathrm{w}}_{e_0}$. Clearly, $E_1$ is an edge of $\Gamma(\tilde f_1)$. Since
$$(\tilde f_1)_{E_1}=(\tilde f_1)^{\mathrm{w}}_{e_0}=monomial\times\prod_{i=1}^{s}(a_{i}x^{m_0}-b_{i}y^{n_0-m_0})^{r_{i}}$$
and since $(\tilde f_1)^{\mathrm{w}}_{e_0+1}=\frac{f^{\mathrm{w}}_{d_0+1}(x_1y_1,y_1)}{y_1^{m}}$, it follows that $f$ is WHND along $E$ iff $\tilde f_1$ is also WHND along $E_1$.
Hence the induction step is proven by applying the induction hypothesis to $\tilde f_1$.
\end{proof}
The above lemma yields that $Q_{E,i}:=(b_{i}:a_{i}), i=1,\ldots,s$, are determined by $f_E$ and they are tangent directions of $\tilde f_{n}$. Then they are infinitely near points in the first neighbourhood of $P_n$. Consequently, they are infinitely near points in the $(n+1)$-th neighbourhood of $0$. To compute the multiplicity $m_{Q_{E,i}}$, we consider the local equation of the strict transform $\tilde f_{n+1}$ of $\tilde f_{n}$ at $Q_{E,i}=(b_{i}:a_{i})$ in chart 2:
\begin{eqnarray*}
\tilde f_{n+1}(u_1,v_1)&=&\frac{\tilde f_{n}((u_1+\frac{b_{i}}{a_{i}})v_1,v_1)}{v_1^{e_0}}\\
&=& \frac{(\tilde f_{n})^{\mathrm{w}}_{e_0}((u_1+\frac{b_{i}}{a_{i}})v_1,v_1)}{v_1^{e_0}}+\frac{(\tilde f_{n})^{\mathrm{w}}_{e_0+1}((u_1+\frac{b_{i}}{a_{i}})v_1,v_1)}{v_1^{e_0}}+\ldots\\
&=& (\tilde f_{n})^{\mathrm{w}}_{e_0}(u_1+\frac{b_{i}}{a_{i}},1)+v_1\cdot (\tilde f_{n})^{\mathrm{w}}_{e_0+1}(u_1+\frac{b_{i}}{a_{i}},1)+\ldots,
\end{eqnarray*}
where $\tilde f_{n}=(\tilde f_{n})^{\mathrm{w}}_{e_0}+(\tilde f_{n})^{\mathrm{w}}_{{e_0}+1}+\ldots$ with $(\tilde f_{n})^{\mathrm{w}}_{e_0}\neq 0$, is the ($(1,1)$-weighted) homogeneous decomposition of $f$. Since 
$$(\tilde f_{n})_{E_{n}}(u,v)=(\tilde f_{n})^{\mathrm{w}}_{e_0}(u,v)=(a_iu-b_iv)^{r_i}\cdot g(u,v)\text{ with }g(b_{i},a_{i})\neq 0,$$
we get
$$\tilde f_{n+1}(u_1,v_1)=(a_iu_1)^{r_i}\cdot g(u_1+\frac{b_{i}}{a_{i}},1)+v_1\cdot (\tilde f_{n})^{\mathrm{w}}_{e_0+1}(u_1+\frac{b_{i}}{a_{i}},1)+\ldots,$$
with $g(u_1+\frac{b_{i}}{a_{i}},1)$ a unit. In the following, this equality will be used to compare the multiplicity $\mathrm{mt}(\tilde f_{n+1})$ with 1.
\begin{lemma}\label{lm22}
With the above notations,
\begin{itemize}
\item[(a)] if $f$ is WHND along $E$, then $m_{Q_{E,i}}=1$ for all $i$;
\item[(b)] if $f$ is not WHND along $E$, then $m_{Q_{E,i}}>1$ for some $i$.
\end{itemize}
\end{lemma}
\begin{proof}
(a) Since $f$ is WHND along $E$, it follows from Lemma \ref{lm21} that $f_n$ is WHND along $E_n$, i.e. either $r_i=1$ for all $i$ or $(a_{i}u-b_{i}v)$ is not a factor of $(\tilde f_{n})^{\mathrm{w}}_{{e_0}+1}$ for each $r_i>1$. If $r_i=1$, it is easy to see that $m_{Q_{E,i}}=\mathrm{mt}(\tilde f_{n+1}(u_1,v_1))=1$ for all $i$. If $r_i>1$ and $(a_{i}u-b_{i}v)$ is not a factor of $(\tilde f_{n})^{\mathrm{w}}_{{e_0}+1}$. Then $(\tilde f_{n})^{\mathrm{w}}_{{e_0}+1}(b_{i},a_{i})\neq 0$. This implies that $(\tilde f_{n})^{\mathrm{w}}_{e_0+1}(u_1+\frac{b_{i}}{a_{i}},1)$ is a unit. Hence $$m_{Q_{E,i}}=\mathrm{mt}(\tilde f_{n+1}(u_1,v_1))=1.$$

(b) Assume that $f$ is not WHND along $E$. By Lemma \ref{lm21}, $\tilde f_{n}$ is not WHND along $E_{n}$, i.e. there exists an $i$ such that $r_i>1$ and $(a_{i}u-b_{i}v)$ is a factor of $(\tilde f_{n})^{\mathrm{w}}_{{e_0}+1}$. Therefore $(\tilde f_{n})^{\mathrm{w}}_{{e_0}+1}(u,v)=(a_{i}u-b_{i}v)\cdot h(u,v)$ and then
$$(\tilde f_{n})^{\mathrm{w}}_{e_0+1}(u_1+\frac{b_{i}}{a_{i}},1)=(a_iu_1)\cdot h(u_1+\frac{b_{i}}{a_{i}},1).$$ 
Hence $m_{Q_{E,i}}=\mathrm{mt}(\tilde f_{n+1}(u_1,v_1))>1$.
\end{proof}
\begin{lemma}\label{lm23}
With the above notations, if $Q$ is not special, then there exists an edge $E$ of $\Gamma(f)$ such that $Q\to Q_{E,i}$ for some $i$.
\end{lemma}
\begin{proof}
We will prove the lemma by induction on $n(f,Q)$. First, since $Q$ is not special, $n(f,Q)\geq 1$. If $n(f,Q)=1$, then $Q$ is a tangent direction of $f$ and we can write $Q=(b:a)$, where $(ax-by)$ is a factor of the tangent cone $f_m$ of $f$. Since $Q$ is not special, $f_m$ is not monomial. This implies that there exists an edge $E$ of $\Gamma(f)$ such that $f_E=f_m$. We can write
$$f_E=f_m=monomial\times \prod_{i=1}^s (a_ix-b_iy)^{r_i}$$ 
with $(b:a)=(b_1:a_1)$, consequently $Q=Q_{E,1}$.

Now we prove the induction step. Suppose that $n(f,Q)>1$. Then $Q\to P$ for some infinitely near point $P$ in the first neighbourhood of $0$. If $P$ is not special, then as above, $P=Q_{E,1}$ for some edge $E$ of $\Gamma(f)$ and hence $Q\to Q_{E,1}$. If $P$ is special, we may assume that $P=(0:1)$. Then the local equation of the strict transform $\tilde f$ of $f$ at $P$ in chart 2, is:
$$\tilde f(u,v)=\frac{f(uv,v)}{v^m}.$$
Since $n(\tilde f,Q)=n(f,Q)-1$ and by induction hypothesis, there is an edge $E'$ of $\Gamma(\tilde f)$ such that
$$\tilde f_{E'}=monomial\times\prod_{i=1}^{s}(a_{i}u^{m'_0}-b_{i}v^{n'_0})^{r_{i}},$$ 
where $a_{i},b_{i}\in K^*$, $(a_{i}:b_{i})$ pairwise distinct; $m'_0,n'_0,r_{i}\in \Bbb N_{>0}$, $\text{gcd}(m'_0,n'_0)=1$ and $Q\to Q_{E',i}$ for some $i$. Let $m_0=m'_0, n_0=m'_0+n'_0$ and let $f=f^{\mathrm{w}}_{d}+f^{\mathrm{w}}_{d+1}+\ldots$ be the $(n_0,m_0)$-weighted homogeneous decomposition of $f$. Then for each $l>d$, we have
\begin{eqnarray*}
\frac{f^{\mathrm{w}}_{l}(uv,v)}{v^m}&=&\sum_{n_0\alpha +m_0\beta =l} c_{\alpha\beta}(uv)^{\alpha}v^{\beta -m}\\
&=&\sum_{n'_0\alpha+m'_0(\alpha+\beta -m)=l-mm'_0} c_{\alpha\beta}u^{\alpha}v^{\alpha+\beta-m}.
\end{eqnarray*}
This implies that $\tilde f=\tilde f^{\mathrm{w}}_{e}+\tilde f^{\mathrm{w}}_{e+1}+\ldots$ is the $(n'_0,m'_0)$-weighted homogeneous decomposition of $\tilde f$, where
$e=d-mm'_0$ and $\tilde f^{\mathrm{w}}_{l-mm'_0}=\frac{f_{l}(uv,v)}{v^m}$. Note that $\tilde f^{\mathrm{w}}_e=\tilde f_{E'}$. It is easy to see that
$$f^{\mathrm{w}}_d(x,y)=y^m \tilde f^{\mathrm{w}}_e(\frac{x}{y},y)=monomial\times\prod_{i=1}^{s}(a_{i}x^{m_0}-b_{i}y^{n_0})^{r_{i}}.$$
Since $E'$ is an edge of $\Gamma(\tilde f_{E'})$, $\tilde f_{E'}$ and then $f^{\mathrm{w}}_d(x,y)$ are not monomials. By $E$ we denote the convex hull of the support of $f^{\mathrm{w}}_d$. Then $E$ is an edge of $\Gamma(f)$ and $f_E=f^{\mathrm{w}}_d$. Therefore $Q_{E,i}=Q_{E',i}$ and hence $Q\to Q_{E,i}$.
\end{proof}
\begin{proof}[Proof of Theorem \ref{thm21}]
($\Longrightarrow$): Assume $f$ is not WHNND, then $f$ is not WHND along some edge $E$ of $\Gamma(f)$. By Lemma \ref{lm21}, there is an infinitely near point $Q_{E,i}$ of $0$, such that $m_{Q_{E,i}}>1$. Clearly, $Q_{E,i}$ is not special. It then follows from Proposition \ref{pro21} that
$$\delta(f)>\nu(f)=\delta_N(f).$$

($\Longleftarrow $): Assume now that $f$ is WHNND. To show $\delta(f)=\delta_N(f)$, it suffices to show that there is no infinitely near point $Q$ of $0$ such that $Q$ is not special and $m_Q>1$. We argue by contradiction. Suppose that there is such an infinitely near point $Q$. By Lemma \ref{lm23}, there is an edge $E$ of $\Gamma(f)$ such that $Q\to Q_{E,i}$ for some $i$, and then $m_{Q}\leq m_{Q_{E,i}}$. Since $f$ is WHND along $E$, it follows from Lemma \ref{lm22} that $m_{Q_{E,i}}=1$. Hence $m_{Q}\leq m_{Q_{E,i}}=1,$
which is a contradiction.  
\end{proof}

If $char(K)=0$ we have Milnor's famous formula $\mu(f)=2\delta(f)-r(f)+1$, where $r(f)$ is the number of branches of $f$. The formula is wrong in general if $char(K)>0$ but still holds if $f$ is NND by \cite[Thm. 4.13]{BGM10}. Using the general inequality
$$\mu_N(f)=2\delta_N(f)-r_N(f)+1\leq 2\delta(f)-r(f)+1\leq \mu(f)$$
from \cite{BGM10}, then Theorem \ref{thm11}, Proposition \ref{pro20} and Proposition \ref{pro22} imply

\begin{corollary}
Let $f\in K[[x,y]]$ be reduced. Then $f$ is INND if and only if $f$ is WNND and $\mu(f)=2\delta(f)-r(f)+1$.
\end{corollary}

\begin{remark}{\rm
\begin{itemize}
\item[(1)] The difference $\mathrm{wvc}(f):=\mu(f)-2\delta(f)+r(f)-1$ counts the number of {\em wild vanishing cycles} of (the Milnor fiber) of $f$  (cf. \cite{Del73},  \cite{MW01}, \cite{BGM10}), which vanishes if $char(K)=0$ or if $f$ is INND.

\item[(2)] $\mathrm{wvc}(f)$ is computable for any given $f$. This follows since $\mu(f)$ is computable by a standard basis computation w.r.t. ~a local ordering (cf. \cite{GP08}) and $\delta(f)$ and $r(f)$ are computable by computing a Hamburger-Noether expansion (cf. \cite{Cam80}). Both algorithms are implemented in {\sc Singular}  (cf. \cite{GPS05}). 
\end{itemize}
}\end{remark}

\begin{example}{\rm
 Consider $f=x(x-y)^2+y^{7}$ and $g=x(x-y)^2+y^{7}+x^6$ and $\mathrm{char}\ (K)=3$. Using 
 {\sc Singular} we compute
$\mu(f)=8, \delta(f)=5,r(f)=3$ and $\mu(g)=8, \delta(f)=4,r(g)=2$. We have $\mathrm{wvc}(f)=0, \mathrm{wvc}(g)=1,\Gamma(f)=\Gamma(g)$ and $f$ is not INND. This shows
\begin{itemize}
\item INND is sufficient but not necessary for the absence of wild vanishing cycles,
\item the Newton diagram can not distinguish between singularities which have wild vanishing cycles and those which have not.
\end{itemize}
}\end{example}

Although we can compute the number of wild vanishing cycles, it seems hard to understand them.   
We like to pose the following

\begin{problem}
Is there any "geometric" way to understand the wild vanishing cycles, distinguishing them from the ordinary vanishing cycles counted by $2\delta-r+1$? Is there at least a "reasonable" characterization of those singularities without wild vanishing cycles?
\end{problem}

\subsubsection*{Acknowledgments}
We would like to thank the referee for his comments and suggestions which prompted us to improve the paper.


\begin{thebibliography}{99}
\baselineskip=16pt
\bibitem[BeP00]{BeP00} P. Beelen and R. Pellikaan, {\em The Newton polygon of plane curves with
many rational points,} Designs, Codes and Cryptography 21(2000), 41-67.

\bibitem[BGM10]{BGM10} Y. Boubakri, G.-M. Greuel, and T. Markwig, {\em Invariants of Hypersurface Singularities in Positive Charecteristic,} Rev. Mat. Complut. (2010) DOI 10.1007/s13163-010-0056-1, 23 pages.

\bibitem[Biv09]{Biv09} C. Bivi\`a-Ausina, {\em Local \L ojasiewicz exponents, Milnor numbers and mixed multiplicities of ideals,} Math. Z. 262(2009),  no. 2, 389-409.

\bibitem[BrK86]{BrK86} E. Brieskorn; H. Kn\"orrer, {\em Plane Algebraic Curves,} Birkhaeuser (1986), 721 pages.

\bibitem[Cam80]{Cam80} A. Campillo, {\em Algebroid Curves in Positive Characteristic,}  Lecture Notes in Math. vol. 613, Springer-Verlag (1980), 168 pages.

\bibitem[Del73]{Del73}P. Deligne, {\em La formule de Milnor, S\'em. G\'eom. Alg\'ebrique du Bois-Marie,} 1967-1969,
SGA 7 II, Lecture Notes in Math. 340, Expose XVI, 197-211 (1973), 1973.

\bibitem[GP08]{GP08} G.-M. Greuel, G. Pfister, {\em A {\sc Singular} Introduction to Commutative Algebra, 2nd Edition,}
Springer (2008), 702 pages.

\bibitem[GPS05]{GPS05} G.-M. Greuel, G. Pfister, H. Sch\"onemann, {\em {\sc Singular 3.1.0}, A Computer Algebra System for Polynomial Computations.} Centre for Computer Algebra, University
of Kaiserslautern (2005). {\bf http://www.singular.uni-kl.de}.

\bibitem[GLS06]{GLS06} G.-M. Greuel, C. Lossen, and E. Shustin, {\em Introduction to Singularities and Deformations,}
Math. Monographs, Springer-Verlag (2006), 476 pages.

\bibitem[Kou76]{Kou76} A. G. Kouchnirenko, {\em Poly\` edres de Newton et nombres de Milnor,} Invent. Math. 32
(1976), 1-31.

\bibitem[Lu87]{Lu87} I. Luengo, {\em The $\mu$-constant stratum is not smooth,} Invent. Math., 90 (1987), 139-152. 

\bibitem[MHW01]{MW01} A. Melle-Hern\'andez and C. T. C. Wall, \emph{Pencils of curves on
  smooth surfaces}, Proc. Lond. Math. Soc., III. Ser. 83 (2001), no.~2, 257--278.
  
\bibitem[Wal99]{Wal99} C. T. C. Wall, {\em Newton polytopes and non-degeneracy,} J. reine angew. Math. 509 (1999), 1-19.

\end{thebibliography}
\end{document}